\documentclass{amsart}%
\usepackage{amssymb}
\usepackage{amsmath}
\usepackage{latexsym}
\usepackage{hyperref}
\usepackage{amsfonts}
\usepackage{graphicx}%
\newtheorem{theorem}{Theorem}

\newtheorem{corollary}[theorem]{Corollary}

\newcommand{\be}{\begin{equation}}
\newcommand{\ee}{\end{equation}}
\newcommand{\bea}{\begin{eqnarray}}
\newcommand{\eea}{\end{eqnarray}}

\begin{document}
\title{Some elementary consequences of Perelman's canonical neighborhood theorem}
\author{Bennett Chow}
\author{Peng Lu$^{1}$}
\maketitle

%The real French abstract:  Nous montrons que les r¨¦cents travaux de Ni et les rendements
%Wilking le r¨¦sultat d'un non mats non compact shrinker Ricci a tout
%au plus la carie courbure scalaire quadratique. Les exemples de
%r¨¦ducteurs de K\"{a}hler-Ricci non compact par Feldman, Ilmanen, et
%pr¨¦sentent Knopf que ce r¨¦sultat est forte. Nous prouvons aussi un
%r¨¦sultat similaire pour certains non compact stable solitons de
%Ricci gradient.

In this purely expository note,\footnotetext[1]{Addresses. Bennett Chow: Math.
Dept., UC San Diego; Peng Lu: Math. Dept., U of Oregon.} we recall a few known
direct consequences of Perelman's canonical neighborhood theorem for
$3$-dimensional Ricci flow and compactness theorem for $3$-dimensional
$\kappa$-solutions. These corollaries regard elementary properties of
$3$-dimensional singularity models and $\kappa$-solutions. Throughout this
note, convergence is in the pointed $C^{\infty}$ Cheeger--Gromov sense,
$\cong$ denotes diffeomorphic, $\operatorname{Rm}$ denotes the Riemann
curvature tensor, and $i\in\mathbb{N}$.

Given $\kappa>0$, a complete solution $\left(  \mathcal{M}^{n},\tilde
{g}(t)\right)  $, $t\in(-\infty,0]$, of the Ricci flow is called a $\kappa
$-solution if $\tilde{g}(t)$, $t\in(-\infty,0]$, is nonflat with uniformly
bounded nonnegative curvature operator $\operatorname{Rm}\geq0$ and $\kappa
$-noncollapsed at all scales.

\begin{corollary}
\label{thm any limit large curv exists}If $\left(  \mathcal{M}^{3},g\left(
t\right)  \right)  $, $t\in\lbrack0,T)$, $T<\infty$, is a singular solution to
the Ricci flow on a closed $3$-manifold, then for any $\left(  x_{i}%
,t_{i}\right)  \in\mathcal{M}\times\lbrack0,T)$ with scalar curvature
$R_{i}\doteqdot R\left(  x_{i},t_{i}\right)  \rightarrow\infty$, there exists
a subsequence of $\left(  \mathcal{M},\tilde{g}_{i}\left(  t\right)  ,\left(
x_{i},0\right)  \right)  $, where $\tilde{g}_{i}\left(  t\right)  \doteqdot
R_{i}g\left(  t_{i}+R_{i}^{-1}t\right)  $, converging to a $\kappa$-solution
$\left(  \mathcal{M}_{\infty}^{3},g_{\infty}\left(  t\right)  ,\left(
x_{\infty},0\right)  \right)  $. In particular, $\left\vert \operatorname{Rm}%
_{g_{\infty}\left(  t\right)  }\right\vert \leq C$ on $\mathcal{M}_{\infty
}\times(-\infty,0]$ for some $C<\infty$.
\end{corollary}

\begin{proof}
By Perelman's improved no local collapsing theorem \cite{Perelman1} (see also
\cite{Topping}), there exists $\kappa>0$ such that if $x_{0}\in\mathcal{M}$,
$t_{0}\in\lbrack0,T)$, and $r_{0}\in(0,1)$ are such that $R\leq r_{0}^{-2}$ in
$B_{g\left(  t_{0}\right)  }\left(  x_{0},r_{0}\right)  $, then $\frac
{\operatorname{Vol}{}_{g\left(  t_{0}\right)  }B_{g\left(  t_{0}\right)
}\left(  x_{0},r\right)  }{r^{n}}\geq\kappa$ for $r\in(0,r_{0}]$%
.\footnote[2]{Thus noncompact singularity models with $R$ bounded have at
least linear volume growth.} Perelman's canonical neighborhood Theorem 12.1 in
\cite{Perelman1} (see also \cite{Kleiner Lott}) says that for $j\in\mathbb{N}%
$, there exists $r_{j}\in(0,1]$ such that if $i_{j}$ is chosen large enough so
that $R_{i_{j}}\geq r_{j}^{-2}$, then $\tilde{g}_{i_{j}}\left(  t\right)  $ on
$B_{\tilde{g}_{i_{j}}\left(  0\right)  }(x_{i_{j}},j^{1/2})\times\lbrack-j,0]$
is $\frac{1}{j}$-close to the corresponding subset of a $\kappa$-solution
$(\mathcal{N}_{j}^{3},h_{j}\left(  t\right)  )$ centered at $y_{j}%
\in\mathcal{N}_{j}$. From Perelman's compactness Theorem 11.7 in
\cite{Perelman1} for $3$-dimensional $\kappa$-solutions and since
$\lim_{j\rightarrow\infty}R_{h_{j}}\left(  y_{j},0\right)  =\lim
_{j\rightarrow\infty}R_{\tilde{g}_{i_{j}}}\left(  x_{i_{j}},0\right)  =1$, we
have that $(\mathcal{N}_{j},h_{j}\left(  t\right)  ,\left(  y_{j},0\right)  )$
subconverges to a $\kappa$-solution $\left(  \mathcal{N}_{\infty}%
^{3},h_{\infty}\left(  t\right)  ,\left(  y_{\infty},0\right)  \right)  $,
$t\in(-\infty,0]$. Since $\frac{1}{j}\rightarrow0$ as $j\rightarrow\infty$, we
conclude that $\left(  \mathcal{M},\tilde{g}_{i_{j}}\left(  t\right)
,(x_{i_{j}},0)\right)  $ subconverges to $\left(  \mathcal{N}_{\infty
},h_{\infty}\left(  t\right)  ,\left(  y_{\infty},0\right)  \right)  $,
$t\in(-\infty,0]$.
\end{proof}

A singularity model of a singular solution $\left(  \mathcal{M}^{n},g\left(
t\right)  \right)  $, $t\in\lbrack0,T)$, $T<\infty$, on a closed manifold is a
complete nonflat ancient solution which is the limit of $g_{i}\left(
t\right)  \doteqdot K_{i}g\left(  t_{i}+K_{i}^{-1}t\right)  $ for some
$t_{i}\rightarrow T$ and $K_{i}\rightarrow\infty$.

\begin{corollary}
\label{OH SO COOL, SO RAD, I LOVE THIS.}Any $3$-dimensional singularity model
must be a $\kappa$-solution.
\end{corollary}

\begin{proof}
Let $\left(  \mathcal{M}_{\infty}^{3},g_{\infty}\left(  t\right)  \right)  $,
$t\in(-\infty,0]$, be a singularity model of a singular solution $\left(
\mathcal{M}^{3},g\left(  t\right)  \right)  $, $t\in\lbrack0,T)$, $T<\infty$,
on a closed $3$-manifold. Then there exist $\left(  x_{i},t_{i}\right)  $ and
$K_{i}\rightarrow\infty$ such that $\left(  \mathcal{M},g_{i}\left(  t\right)
,(x_{i},0)\right)  $ converges to $\left(  \mathcal{M}_{\infty}^{3},g_{\infty
}\left(  t\right)  ,(x_{\infty},0)\right)  $, where $g_{\infty}\left(
t\right)  $ is nonflat. Let $R_{i}\doteqdot R_{g}\left(  x_{i},t_{i}\right)
$. Since $\lim_{i\rightarrow\infty}K_{i}^{-1}R_{i}=R_{g_{\infty}}\left(
x_{\infty},0\right)  \doteqdot c$ exists, where $c>0$ by the strong maximum
principle, we also have that $\tilde{g}_{i}\left(  t\right)  \doteqdot
R_{i}g\left(  t_{i}+R_{i}^{-1}t\right)  $ converges to $\left(  \mathcal{M}%
_{\infty},\tilde{g}_{\infty}\!\left(  t\right)  \!,\!(x_{\infty},0)\right)  $,
$t\in(-\infty,0]$, where $\tilde{g}_{\infty}\left(  t\right)  =cg_{\infty
}\left(  c^{-1}t\right)  $. By Corollary \ref{thm any limit large curv exists}%
, there exists $\left\{  i_{j}\right\}  $ such that $\left(  \mathcal{M}%
,\tilde{g}_{i_{j}}\left(  t\right)  ,(x_{i_{j}},0)\right)  $ converges to a
$\kappa$-solution $\left(  \mathcal{N}_{\infty}^{3},h_{\infty}\left(
t\right)  ,\left(  y_{\infty},0\right)  \right)  $, $t\in(-\infty,0]$. Hence
$\left(  \mathcal{M}_{\infty},\tilde{g}_{\infty}\left(  t\right)  \right)  $
is isometric to $\left(  \mathcal{N}_{\infty},h_{\infty}\left(  t\right)
\right)  $ on $(-\infty,0]$. Thus $g_{\infty}\left(  t\right)  =c^{-1}%
\tilde{g}_{\infty}\left(  ct\right)  $ is a $\kappa$-solution.
\end{proof}

Let $\left(  \mathcal{N}^{3},h\right)  $ be a complete Riemannian
$3$-manifold. Given $\varepsilon>0$ and $p\in\mathcal{N}$ with $R\left(
p\right)  >0$, a geodesic ball $B(p,\varepsilon^{-1}r)$ is called an
$\varepsilon$-neck if $r^{-2}h$ on $B(p,\varepsilon^{-1}r)$ is $\varepsilon
$-close in the $C^{\left\lceil \varepsilon^{-1}\right\rceil +1}$-topology
(here $\left\lceil \varepsilon^{-1}\right\rceil $ denotes the least integer
$\geq\varepsilon^{-1}$) to a piece of the unit cylinder $g_{\operatorname{cyl}%
}=g_{\mathcal{S}^{2}}+du^{2}$ on $\mathcal{S}^{2}\times\mathbb{R}$.

We have the following (which may also be proved using Corollary 9.88 in
\cite{Mogan-Tian}).

\begin{corollary}
\label{Actually is this true? What is the proof? Obvious? Nontrivial? Hard?}%
The asymptotic cone of a noncompact orientable $3$-dimensional $\kappa
$-solution must be either a line or a half-line.
\end{corollary}

\begin{proof}
Let $(\mathcal{M}^{3},g\left(  t\right)  )$, $t\in(-\infty,0]$, be a
noncompact orientable $3$-dimensional $\kappa$-solution. Since
$\operatorname{Rm}\geq0$, by the strong maximum principle and Hamilton's
classification of $2$-dimensional $\kappa$-solutions, $(\mathcal{M}%
^{3},g\left(  t\right)  )$ is isometric to:

\begin{enumerate}
\item $\mathcal{S}^{2}\times\mathbb{R}$ or $(\mathcal{S}^{2}\times
\mathbb{R})/\mathbb{Z}_{2}$, where $\mathcal{S}^{2}$ is the shrinking round
$2$-sphere, or

\item a noncompact $\kappa$-solution with $\operatorname{Rm}>0$ and
$\mathcal{M}\cong\mathbb{R}^{3}$.
\end{enumerate}

\noindent In case (i), the asymptotic cone of $(\mathcal{M},g\left(  t\right)
)$ is either a line or a half-line. In case (ii) it suffices to prove the
asymptotic cone of $g\left(  0\right)  $, which exists since
$\operatorname{sect}\geq0$, is a half-line; here $\operatorname{sect}$ denotes
the sectional curvature.\smallskip

\textbf{Claim 1.} \emph{For any }$\varepsilon>0$\emph{ there exists an
}$\varepsilon$\emph{-neck} $\mathcal{N}_{\varepsilon}$ \emph{contained in}
$(\mathcal{M}^{3},g\left(  0\right)  )$.\smallskip

Recall from Proposition 11.4 of \cite{Perelman1} that the asymptotic scalar
curvature ratio $\operatorname{ASCR}(g(t))=\infty$ for $t\in(-\infty,0]$.
Hence, by dimension reduction for noncompact $\kappa$-solutions with
$\operatorname{ASCR}=\infty$, there exists a sequence $\left\{  x_{i}\right\}
$ in $\mathcal{M}$ such that $g_{i}\left(  t\right)  =R_{i}g\left(  R_{i}%
^{-1}t\right)  $, where $R_{i}\doteqdot R_{g}(x_{i},0)$, on $\mathcal{M}%
\times(-\infty,0]$ and based at $(x_{i},0)$, converges to the product of
$\mathbb{R}$ with a $2$-dimensional $\kappa$-solution, which must be the
shrinking round $\mathcal{S}^{2}$. The existence of $\varepsilon$-necks in
$(\mathcal{M},g\left(  t\right)  )$, for any $\varepsilon>0$ and any
$t\in(-\infty,0]$, now follows from the definition of convergence. This proves
Claim 1.\smallskip

Since $\mathcal{N}_{\varepsilon}\cong\mathcal{S}^{2}\times\mathbb{R}$ and
$\partial\mathcal{N}_{\varepsilon}$ is embedded in $\mathcal{M}\cong
\mathbb{R}^{3}$, by the smooth Sch\"{o}nflies theorem we have that
$\mathcal{M}-\mathcal{N}_{\varepsilon}$ has exactly two components, a compact
component $\mathcal{B}_{\varepsilon}$ diffeomorphic to a closed $3$-ball and a
noncompact component $\mathcal{C}_{\varepsilon}\cong\mathcal{S}^{2}%
\times\lbrack0,1)$.

Since $\mathcal{N}_{\varepsilon}$ is an $\varepsilon$-neck, there exist an
embedding $\psi_{\varepsilon}:\mathcal{S}^{2}\times\left[  -\varepsilon
^{-1}+4,\varepsilon^{-1}-4\right]  \rightarrow\mathcal{N}_{\varepsilon}$ and
$r_{\varepsilon}>0$ such that $r_{\varepsilon}^{-2}\psi_{\varepsilon}^{\ast
}g(0)$ is $\varepsilon$-close in the $C^{\left\lceil \varepsilon
^{-1}\right\rceil +1}$-topology to $g_{\operatorname{cyl}}$. We may assume
that $\psi_{\varepsilon}(\mathcal{S}^{2}\times\{-\frac{1}{\varepsilon}+4\})$
is closer to $\partial\mathcal{B}_{\varepsilon}$ and $\psi_{\varepsilon
}(\mathcal{S}^{2}\times\{\frac{1}{\varepsilon}-4\})$ is closer to
$\partial\mathcal{C}_{\varepsilon}$. Since $g\left(  0\right)  $ has bounded
curvature, $r_{\varepsilon}\geq c>0$, independent of $\varepsilon$
small.\smallskip

\textbf{Claim 2.} \emph{From now on, fix} $O\in\mathcal{M}$. \emph{For
}$\varepsilon>0$\emph{ sufficiently small, }$O\in\mathcal{B}_{\varepsilon}%
$.\smallskip

Since $\lim_{\varepsilon\rightarrow0}\max_{x\in\mathcal{N}_{\varepsilon}%
}\left(  \min\left\{  \operatorname{sect}_{g\left(  0\right)  }\left(
P_{x}\right)  :P_{x}\text{ is a }2\text{-plane at }x\right\}  \right)  =0$
($g_{\operatorname{cyl}}$ has a $0$ sectional curvature everywhere) and since
$\operatorname{Rm}_{g\left(  0\right)  }\left(  O\right)  >0$, for
$\varepsilon>0$ small we have that $O\notin\mathcal{N}_{\varepsilon}$. If the
claim is false, then there exist $\varepsilon_{i}\searrow0$ such that
$O\in\mathcal{C}_{\varepsilon_{i}}$ for all $i$. We may pass to a subsequence
$\left\{  k_{i}\right\}  $ with (the $\mathcal{N}_{\varepsilon_{k_{i}}}$ are
pairwise disjoint)%
\begin{equation}
\mathcal{N}_{\varepsilon_{k_{i}}}\subset\mathcal{C}_{\varepsilon_{k_{j}}}%
\quad\text{for }j<i.\label{asymmetric participation}%
\end{equation}
Indeed, suppose we have chosen $1\doteqdot k_{1}<\cdots<k_{i-1}$. Since
$\mathcal{K}_{i}\doteqdot%
%TCIMACRO{\tbigcup _{j<i}}%
%BeginExpansion
{\textstyle\bigcup_{j<i}}
%EndExpansion
(\overline{\mathcal{M}-\mathcal{C}_{\varepsilon_{k_{j}}}})$ is compact,
$\operatorname{sect}_{g\left(  0\right)  }$ has a positive lower bound on
$\mathcal{K}_{i}$. By this and $\lim_{i\rightarrow\infty}\varepsilon_{i}=0$,
we conclude that there exists $k_{i}>k_{i-1}$ such that $\mathcal{N}%
_{\varepsilon_{k_{i}}}\cap\mathcal{K}_{i}=\varnothing$, which implies
(\ref{asymmetric participation}).

By again using the smooth Sch\"{o}nflies theorem, for each $i$ and $j$ with
$j<i$ we have $\mathcal{C}_{\varepsilon_{k_{j}}}-\mathcal{N}_{\varepsilon
_{k_{i}}}\doteqdot\mathcal{K}_{i,j}\cup\mathcal{L}_{i,j}$, where
$\mathcal{K}_{i,j}\cong\mathcal{S}^{2}\times\lbrack0,1)$ has compact closure
in $\mathcal{M}$ and $\mathcal{L}_{i,j}\cong\mathcal{S}^{2}\times\lbrack0,1)$
satisfies $\overline{\mathcal{L}_{i,j}}=\mathcal{L}_{i,j}$. Since
$\mathcal{M}^{3}\cong\mathbb{R}^{3}$, we conclude $\mathcal{B}_{\varepsilon
_{k_{j}}}\subset\mathcal{B}_{\varepsilon_{k_{i}}}$ and $\mathcal{C}%
_{\varepsilon_{k_{j}}}\supset\mathcal{C}_{\varepsilon_{k_{i}}}$ for all $j<i$.
Claim 2 follows from:\smallskip

\textbf{Subclaim.} $%
%TCIMACRO{\dbigcap \nolimits_{i\in\mathbb{N}}}%
%BeginExpansion
{\displaystyle\bigcap\nolimits_{i\in\mathbb{N}}}
%EndExpansion
\,\mathcal{C}_{\varepsilon_{k_{i}}}=\varnothing$.\smallskip

\noindent Fix $p\in\mathcal{B}_{\varepsilon_{k_{1}}}$, so that $p\in
\mathcal{B}_{\varepsilon_{k_{i}}}$ for all $i$. Suppose the subclaim is false;
then there exists $x\in\,\mathcal{C}_{\varepsilon_{k_{i}}}$ for all $i$. Let
$\gamma$ be a minimal geodesic from $p$ to $x$ with respect to $g\left(
0\right)  $. Then $\gamma$ must pass from one end of $\mathcal{N}%
_{\varepsilon_{k_{i}}}$ to the other end. Hence, for $i$ large,%
\[
d_{g\left(  0\right)  }\left(  p,x\right)  =\operatorname{L}{}_{g\left(
0\right)  }\left(  \gamma\right)  \geq\frac{1}{2}\operatorname{diam}_{g\left(
0\right)  }(\mathcal{N}_{\varepsilon_{k_{i}}})\geq\frac{1}{2}\varepsilon
_{k_{i}}^{-1}r_{\varepsilon_{k_{i}}}\geq\frac{c}{2}\varepsilon_{k_{i}}^{-1},
\]
where $c>0$ is independent of $i$. The subclaim follows from $\varepsilon
_{k_{i}}^{-1}\rightarrow\infty$.\smallskip

Now let $\operatorname{Ray}_{\mathcal{M}}\left(  O\right)  $ denote the space
of unit speed rays emanating from $O$ in $\left(  \mathcal{M}^{3},g\left(
0\right)  \right)  $. We have the pseudo-metric $\tilde{d}_{\infty}\left(
\gamma_{1},\gamma_{2}\right)  \doteqdot\lim_{s,t\rightarrow\infty}%
\tilde{\measuredangle}\gamma_{1}\left(  s\right)  O\,\gamma_{2}\left(
t\right)  \in\left[  0,\pi\right]  $ on $\operatorname{Ray}_{\mathcal{M}%
}\left(  O\right)  $, for $\gamma_{1},\gamma_{2}\in\operatorname{Ray}%
_{\mathcal{M}}\left(  O\right)  $ and where $\tilde{\measuredangle}$ is the
Euclidean comparison angle. The asymptotic cone of $\left(  \mathcal{M}%
,g(0)\right)  $ is isometric to the Euclidean metric cone $\operatorname{Cone}%
\left(  \mathcal{M}\left(  \infty\right)  ,d_{\infty}\right)  $, where
$\left(  \mathcal{M}\left(  \infty\right)  ,d_{\infty}\right)  $ is the
quotient metric space induced by $(\operatorname{Ray}_{\mathcal{M}}\left(
O\right)  ,\tilde{d}_{\infty})$. Thus, the conclusion that the asymptotic cone
of $\left(  \mathcal{M},g(0)\right)  $ is a half-line shall follow from
showing that for all $\gamma_{1},\gamma_{2}\in\operatorname{Ray}_{\mathcal{M}%
}\left(  O\right)  $, $\tilde{d}_{\infty}\left(  \gamma_{1},\gamma_{2}\right)
=0$.

Since $O\in\mathcal{B}_{\varepsilon}$ for $\varepsilon>0$ sufficiently small,
we have for such $\varepsilon$ that any $\gamma\in\operatorname{Ray}%
_{\mathcal{M}}\left(  O\right)  $ passes from one end of $\mathcal{N}%
_{\varepsilon}$ to the other end.\smallskip

\textbf{Claim 3.} \emph{For any }$\varepsilon>0$\emph{ sufficiently small, we
have that any} $\gamma\in\operatorname{Ray}_{\mathcal{M}}\left(  O\right)  $
\emph{intersects }$\psi_{\varepsilon}(\mathcal{S}^{2}\times\{\varepsilon
^{-1}-4\})$ \emph{at exactly one point, which we define to be} $\gamma
(t_{\gamma,\varepsilon}).$\smallskip

This follows from the facts that rays minimize and that the geometry of any
$\varepsilon$-neck is, after rescaling, $\varepsilon$-close to that of the
unit cylinder of length $2\varepsilon^{-1}$.\smallskip

We have $\lim\limits_{t\rightarrow\infty}\frac{d_{g\left(  0\right)  }\left(
\gamma_{1}\left(  at\right)  ,\gamma_{2}\left(  bt\right)  \right)  }%
{t}=\left(  a^{2}+b^{2}-2ab\cos\left(  \tilde{d}_{\infty}\left(  \gamma
_{1},\gamma_{2}\right)  \right)  \right)  ^{1/2}$. Hence $\tilde{d}_{\infty
}\left(  \gamma_{1},\gamma_{2}\right)  =0$ if and only if $\lim
\limits_{t\rightarrow\infty}\frac{d_{g\left(  0\right)  }\left(  \gamma
_{1}\left(  t\right)  ,\gamma_{2}\left(  t\right)  \right)  }{t}=0$. Now
$\tilde{d}_{\infty}\left(  \gamma_{1},\gamma_{2}\right)  =0$ follows
from:\smallskip

\textbf{Claim 4.} \emph{For any} $\gamma_{1},\gamma_{2}\in\operatorname{Ray}%
_{\mathcal{M}}\left(  O\right)  $, $\lim_{\varepsilon\rightarrow0}%
\frac{d_{g\left(  0\right)  }\left(  \gamma_{1}\left(  t_{\gamma
_{1},\varepsilon}\right)  ,\gamma_{2}\left(  t_{\gamma_{2},\varepsilon
}\right)  \right)  }{\min\{t_{\gamma_{1},\varepsilon},t_{\gamma_{2}%
,\varepsilon}\}}=0$.

Since $\mathcal{N}_{\varepsilon}$ is an $\varepsilon$-neck and the diameter of
a round $2$-sphere of radius $r$ is $\pi r$, we have that $d_{g\left(
0\right)  }\left(  \gamma_{1}\left(  t_{\gamma_{1},\varepsilon}\right)
,\gamma_{2}\left(  t_{\gamma_{2},\varepsilon}\right)  \right)  \leq2\pi
r_{\varepsilon}$ for $\varepsilon$ sufficiently small. Since $\gamma_{1}$ and
$\gamma_{2}$ are rays emanating from the same point, this implies $\left\vert
t_{\gamma_{1},\varepsilon}-t_{\gamma_{2},\varepsilon}\right\vert \leq2\pi
r_{\varepsilon}$. Since $\left.  \gamma_{1}\right\vert _{[0,t_{\gamma
_{1},\varepsilon}]}$ and $\left.  \gamma_{2}\right\vert _{[0,t_{\gamma
_{2},\varepsilon}]}$ both intersect $\psi_{\varepsilon}(\mathcal{S}^{2}%
\times\{-\varepsilon^{-1}\})$ and $\psi_{\varepsilon}(\mathcal{S}^{2}%
\times\{\varepsilon^{-1}-4\})$, we have $\min\{t_{\gamma_{1},\varepsilon
},t_{\gamma_{2},\varepsilon}\}\geq\varepsilon^{-1}r_{\varepsilon}$. Claim 4
follows easily.
\end{proof}

\markright{CONSEQUENCES OF PERELMAN'S CANONICAL NEIGHBORHOOD
THEOREM}

\end{document}